\documentclass[a4paper]{amsart}
\usepackage[textwidth=15cm,top=3cm, bottom=3cm, hcentering]{geometry}
\usepackage[colorlinks=true,linkcolor=red,citecolor=blue]{hyperref}

\usepackage[utf8]{inputenc}
\usepackage[english]{babel}
\usepackage{hyperref,amsmath,dsfont,amsthm,mathtools,enumerate,xcolor,mathrsfs,enumitem,amssymb,enumitem,graphicx,comment}

\usepackage[all,knot,cmtip]{xy}
\usepackage[titletoc,title]{appendix}
 \usepackage{color}
 
\usepackage{enumerate}
\newcommand{\ot}{\otimes}
 
\newcommand{\Q}{\mathbb{Q}}

\newcommand{\CP}{\mathbb{C}P}

\newcommand{\field}{\mathds{k}}

\renewcommand{\L}{\mathbb L}
\newcommand{\Aa}{\mathcal{A}}

\newcommand{\Oo}{\mathcal O}

\newcommand{\g}{\mathfrak g}

\newcommand{\kk}{\Bbbk}

\newcommand{\Com}{{\mathcal Com}}
\newcommand{\Lie}{{\mathcal Lie}}

\DeclareMathOperator{\id}{id}

\DeclareMathOperator{\map}{map}

\DeclareMathOperator{\aut}{aut}

\DeclareMathOperator{\Hom}{Hom}

\DeclareMathOperator{\Der}{Der}

\DeclareMathOperator{\MC}{MC}

\DeclareMathOperator{\im}{im}

\newtheorem{thm}{Theorem}[section]

\newtheorem{prop}[thm]{Proposition}

\newtheorem{lemma}[thm]{Lemma}

\theoremstyle{definition}
\newtheorem{dfn}[thm]{Definition}
\newtheorem{exmp}[thm]{Example}
\newtheorem{rmk}[thm]{Remark}

\thanks{
J. Cirici acknowledges the Serra H\'{u}nter Program the AEI (CEX2020-001084-M and PID2020-117971GB-C22). This work was also partially supported by the Departament de Recerca i Universitats de la Generalitat de Catalunya (2021 SGR 00697) and the ANR-20-CE40-0016 HighAGT. B. Saleh acknowledges the support by the Knut and Alice Wallenberg Foundation through grant no. 2019.0521. 
}

\address[J. Cirici]{Departament de Matemàtiques i Informàtica, Universitat de Barcelona\\
Gran Via 585\\
08007 Barcelona, Spain  / Centre de Recerca Matemàtica, Edifici C, Campus Bellaterra, 08193 Bellaterra, Spain}
\email{jcirici@ub.edu}

\address[B. Saleh]{Department of Mathematics, Stockholm University, SE-106 91 Stockholm, Sweden}
\email{basharsaleh1@gmail.com}

\begin{document}
\title[Weight decompositions for mapping spaces and homotopy automorphisms]{Weight decompositions on algebraic models\\ for  mapping spaces and homotopy automorphisms} 
\author{Joana Cirici and Bashar Saleh} 
\date{}

\begin{abstract}
We obtain restrictions on the rational homotopy types of mapping spaces and of classifying spaces of homotopy automorphisms by means of the theory of positive weight decompositions. The theory applies, in particular, to connected components of holomorphic maps between compact Kähler manifolds as well as homotopy automorphisms of Kähler manifolds.
\end{abstract}
\maketitle
\section{Introduction}
Weight decompositions on algebraic objects are often useful to limit the range of the structure maps involved, leading to non-trivial homotopical consequences.
A main source of weight decompositions is Deligne's functorial theory of mixed Hodge structures on the rational cohomology of any complex algebraic variety.
Part of the data of a mixed Hodge structure is the weight filtration, an increasing filtration defined at the level of rational vector spaces.
There is a functorial weight decomposition on the Sullivan minimal model of the underlying complex analytic space of any complex algebraic variety in such a way that, in cohomology, it recovers the graded pieces of Deligne's weight filtration. Moreover, for a smooth variety
 the weight decomposition on the model is positive (see \cite{morgan78}; for the functorial statement, see also \cite{CG1} or Theorem 4.9 of \cite{bcc}). In the case of smooth projective varieties or, more generally, compact Kähler manifolds, weights are not only positive but actually pure, a property that is equivalent to having formality \cite{DGMS}. In particular, one may think about the property of having positive weights as a weaker situation than formality, but still homotopically restrictive.

Positive weights on general commutative differential graded algebras were first considered by Body and Douglas in \cite{BD1} and \cite{BD2}, who showed that positive weight decompositions 
lead to unique factorizations of rational homotopy types.
In the related work \cite{BMSS}, Body, Mimura, Shiga and Sullivan used positive weights in order to characterize $p$-universal spaces. We refer to 
\cite{AL},  \cite{MaWe},  \cite{bcc} and \cite{Manin} for more recent applications of positive weights in various geometric and topological contexts.

In an effort to further understand the homotopical consequences of having positive weights, in this paper we deploy weight-graded basic results in homotopy theory.
In particular, we show that if a differential graded algebra carries positive weights and has finite-dimensional cohomology, then 
there is a transferred structure on the cohomology $H$ and some finite number $k\geq 3$ such that the higher $A_\infty$-operations $\mu_n\colon H^{\ot n}\to H$ vanish for $n\geq k$. The same argument shows that the $n$-tuple Massey products on cohomology are all trivial for sufficiently large $n$.

A topological space $X$ is said to have \textit{positive weights} if it admits a model in the sense of Sullivan with a positive weight decomposition: there is a quasi-isomorphism of commutative differential graded algebras $A\to \Aa_{pl}(X)$ such that $A=\bigoplus_{p\geq 0} A_p$ admits a non-negative weight-grading compatible with the differential and products, and such that $A_0$ is concentrated in degree 0. Here $\Aa_{pl}(X)$ denotes Sullivan's algebra of piece-wise linear forms on $X$. Equivalently, we might ask for a Lie model in the sense of Quillen with a negative weight decomposition. Note that formality and coformality both imply positive weights. 
The main goal of this paper is to study the behaviour of weights on the rational homotopy types of mapping spaces and homotopy automorphisms. 
We prove:

\begin{thm}\label{first}
Let $X$ be of the homotopy type of a finite CW-complex and let $Y$ be a connected nilpotent space of finite $\Q$-type having positive weights. Let $\Psi\colon X\to Y$ denote a constant map. Then the connected component of $\map(X,Y)$ which contains $\Psi$,  has positive weights.
\end{thm}

We actually prove a more general result for not necessarily constant maps satisfying a weight-preserving property (see Proposition \ref{posweightsgeneral}) which in the context of complex geometry leads to  the following result:

\begin{thm}\label{second}
Let $f:X\to Y$ be a holomorphic map between compact Kähler manifolds, where $Y$ is nilpotent.
Then the connected component of the mapping space $map(X,Y)$
which contains $f$, has positive weights.
\end{thm} 
The above result is actually true for any formal map between formal topological spaces.
For an arbitrary map of complex algebraic varieties we also obtain restricted weights, but these are not necessarily positive. Such weights may be interpreted as part of a mixed Hodge structure on the models for mapping spaces of complex algebraic varieties. 
Note that the mapping spaces we are considering are purely topological (no algebraic condition except for the fact that we pick a connected component of an algebraic map) and so have a priory no reason to carry mixed Hodge structures. Therefore our results exhibit in some sense the motivic nature of mapping spaces.

In \cite{LuSm}, Lupton and Smith show that for a formal space, the universal cover of the classifying space of the space of its homotopy automorphisms has positive weights. We extend this result to the coformality and pointed setting:

\begin{thm}\label{third}
Let $X$ be a simply connected space of the homotopy type of a finite CW-complex. If $X$ is formal or coformal, then the universal covers of $B\aut(X)$ and $B\aut_{*}(X)$ have positive weights.
\end{thm}

Again, in the case when the initial space is a complex algebraic variety, the above result may be understood as a manifestation of a mixed Hodge theory inherited by homotopy automorphisms.

We also endow algebraic models for algebraic constructions related to the free and based loop spaces with weights inherited from weight decompositions on the original space.

\medskip 

We briefly explain the structure of this paper.
The algebraic models of mapping spaces turn out to be more accessible through Quillen's approach to rational homotopy, via differential graded Lie or $L_\infty$-algebras. With this in mind, we first review the theory of weights in this setting and show how weights are naturally transferred between the Sullivan and Quillen approaches to rational homotopy theory. We do this in Section 2, where we also show that the main constructions in rational homotopy (such as minimal models and homotopy transfer theory) generalize to the weight-graded setting, complementing work of Douglas \cite{Douglas}. 
Section 3 is devoted to the study of particular properties of weights. We review the notions of purity and of having positive weights. We also consider an intermediate situation where weights are restricted to live in a certain segment. All of these restrictions lead to non-trivial homotopical consequences. We also review two main sources of weights: the weight decompositions induced from formality and coformality respectively.
In Section 4 we prove the main theorems of this paper on mapping spaces and classifying spaces of homotopy automorphisms.

\subsection*{Acknowledgments}We would like to thank 
 Geoffroy Horel and Alexander Berglund   for useful discussions and suggestions.
Also, thanks to
José Moreno, Aniceto Murillo and Daniel Tanré for answering our questions on completed Lie algebras.

\subsection*{Conventions}
We will be using a cohomological convention even for the dgla's, and we reserve the lower index for weights. In particular, Lie models, homology, and homotopy groups of spaces will be concentrated in negative cohomological degrees.

\section{Weights in rational homotopy}
In this section we review the notion of weight decomposition on a 
differential graded algebra and promote basic homotopical constructions to the weight-graded context. For practical purposes, we focus on the commutative and Lie settings, although a general operadic approach is also possible. 

\subsection{Weight-graded algebras}
All algebras in this section will be over a field $\field$ of characteristic zero. 
Recall the following definition from \cite{morgan78}, \cite{BD1}.
\begin{dfn}\label{Def_cdga_weight}
Let $(A,d,\cdot)$ be a non-negatively graded dg-algebra over $\kk$.
A \textit{weight decomposition} of $A$
is a  direct sum decomposition
\[A^n=\bigoplus_{p\in\mathbb{Z}} A_p^n\]
of each vector space $A^n$, such that:
\begin{enumerate}
 \item $dA^n_p\subseteq A^{n+1}_p$ for all $n\geq 0$ and all $p\in\mathbb{Z}$.
 \item $A_p^n\cdot A_{q}^{m}\subseteq A_{p+q}^{n+m}$ for all $n,m\geq 0$ and all $p,q\in\mathbb{Z}$.
\end{enumerate}
Given $x\in A^n_p$ we will denote by $|x|=n$ its \textit{cohomological degree} and by $w(x)=p$ its \textit{weight}.
\end{dfn}

\begin{dfn}A weight decomposition on a cdga $A$ is said to be \textit{positive (negative)} if $A^0$ has weight 0 and $A^i$ is concentrated in positive  (negative) degrees for all $i\neq 0$.
\end{dfn}

It follows from the definition that  the cohmology of $A$ has an induced weight decomposition by letting $H^n(A)_p:= H^n(A_p)$.
Also, if $A$ is unital weight-graded, then the unit $1\in A$ is of weight-zero.

\begin{rmk}
The above definition has its obvious adaptation to other operadic algebras.
In the case of a dg Lie algebra (dgla), to have a weight decomposition one just asks that both the differential and the Lie bracket preserve weights, while for dg-coalgebras 
the differential and coproduct should preserve the weights.
In the case of $A_\infty$, $C_\infty$ or $L_\infty$ algebras, 
weight decompositions should be preserved by all structure maps  $\mu_n\colon A^{\ot n}\to A$.
\end{rmk}

\begin{dfn}\label{dualizing}
The dual of a finite type commutative dg-coalgebra (cdgc) $C$ with a weight decomposition gives a commutative dg-algebra (cdga) $C^\vee$ with a weight decomposition, by letting $(C^\vee)_p:=(C_{-p})^\vee$. The same is true for the passage from finite type cdga's to cdgc's.
\end{dfn}

\subsection{Chevalley-Eilenberg and Quillen constructions}\label{subsec:CE-Quillen-constr}
Under some connectivity hypotheses, there is an adjunction between cdgc's and dgla's,
given by the Chevalley-Eilenberg and Quillen constructions
(see for instance \cite[§ 22]{felixrht}). We promote this adjunction to the weight-graded setting.

First, the \textit{Chevalley-Eilenberg coalgebra construction} associates, to any dgla $L$,
a cdgc $\mathcal{C}_{CE}^*(L)$ whose underlying complex is the Chevalley-Eilenberg complex of $L$.
As a graded coalgebra  we have that 
\[\mathcal{C}_{CE}^*(L) := \Lambda^c(s^{-1} L).\]
where $\Lambda^c(V)$ denotes the free graded cocommutative coalgebra generated by the desuspension $s^{-1}V$ of $V$, which  in each degree is given by
$(s^{-1}V)^n:=V^{n+1}.$
The differential on $\mathcal{C}_{CE}^*(L)$ is
 given by  $d=d_\alpha+ d_\beta$, where 
\[
d_\alpha(s^{-1}x_1\wedge \dots \wedge s^{-1}x_n):= \sum_{i}\pm s^{-1}x_1\wedge\cdots \wedge s^{-1}d_Lx_i\wedge \cdots \wedge s^{-1}x_n
\]
and
\[
d_\beta(s^{-1}x_1\wedge \dots \wedge s^{-1}x_n):=\sum_{i,j}\pm s^{-1}[x_i,x_j]\wedge s^{-1}x_1\wedge\cdots \wedge \widehat{s^{-1}x_i}\wedge \cdots \wedge \widehat{s^{-1}x_j}\wedge \cdots \wedge s^{-1}x_n.
\]
If the dgla $L$ carries a weight decomposition, we obtain a weight decomposition on the cdgc $\mathcal{C}_{CE}^*(L)$ by setting 
\[w(s^{-1}x_1\wedge \dots \wedge s^{-1}x_n):=w(x_1)+\dots+w(x_n).\]
The Chevalley-Eilenberg coalgebra construction is also defined for an $L_\infty$-algebra $(L,\ell_i)$, with the differential on $\mathcal{C}_{CE}^*(L)$ given by 
$$
d(s^{-1}x_1\wedge \dots \wedge s^{-1}x_n) =\sum_{m\leq n}\sum_{i_1,\dots, i_m}\pm\ell_m(x_{i_1},\dots\, x_{i_m})\wedge s^{-1}x_1\wedge\cdots \wedge \widehat{s^{-1}x_{i_1}}\wedge \cdots \wedge \widehat{s^{-1}x_{i_m}}\wedge \cdots \wedge s^{-1}x_n.
$$
This differential is weight preserving provided that the $L_\infty$-algebra carries a weight decomposition.

By Definition \ref{dualizing}, dualizing this construction,
the \textit{Chevalley-Eilenberg cdga} \[\Aa^*_{CE}(L) := \mathcal{C}_{CE}^*(L)^\vee\]
of a dgla of $L_\infty$-algebra with a weight decomposition,
inherits a weight decomposition. Note that this is only defined for dgla's and $L_\infty$-algebras of finite type.

Given a counital cdgc $C$ we obtain a non-counital cdgc 
$\overline{C}:= C/\field \mathbf{1}$ on which the reduced coproduct $\bar \Delta\colon\overline{C}\to\overline{C}\ot\overline{C}$ is given by $\bar \Delta(c)=\Delta(c)-c\ot \mathbf{1}-\mathbf{1}\ot c$.
The Quillen construction associates to $C$ the dgla $\mathscr L^\star(C)$, given as a graded object by 
$$
\mathscr L^\star(C):= 
\L(s\overline{C}),
$$
where $\L V$ denotes the free graded Lie algebra generated by $V$.
Its differential is the unique derivation on $\mathscr L^\star(C)$ that satisfies
$$d(sc) =-sd_C(c) + [-,-]\circ (s\ot s) \circ \bar \Delta(c)$$
for every $c\in\overline C$. 
If $C$ has a weight decomposition, we obtain a weight decomposition 
on $\mathscr L^\star(C)$ by letting
$$w([sc_1,[\cdots,[sc_{n-1},sc_n]\cdots]):=w(c_1)+\dots+w(c_k).$$

By Definition \ref{dualizing}, if $A$ is a graded unital cdga of finite type with a weight decomposition 
then its \textit{Quillen dgla construction} 
\[\mathscr L(A) := \mathscr L^\star(A^\vee)\]
has a weight decomposition.


The following is a promotion to the weight-graded setting of 
the well-known adjunction between the Chevalley-Eilenberg and Quillen constructions. The classical proof (see for instance \cite{lodayvallette}) adapts mutatis mutandis to algebras with weights after considering weight-preserving twisting morphisms:

\begin{prop}\label{prop:koszul-duality}
Let $A$ be a cdga of finite type and $L$ an $L_\infty$-algebra, both with weight decompositions.
There are weight preserving quasi-isomorphisms
\[q_a\colon \Aa^*_{CE}(\mathscr L(A))\to A
\text{ and }
q_\ell\colon \mathscr L^\star(\mathcal{C}_{CE}^*(L))\to L.
\]
\end{prop}

\subsection{Weights on algebraic models}
Recall that a cdga $A$ is said to be \textit{connected} if it is unital and $\bar A = A/\field \mathbf{1}$ is concentrated in positive cohomological degrees. 
The homotopy theory of cdga's has its obvious adaptation to the weight-graded setting. We review some main results on minimal models for cdga's and dgla's respectively.

\begin{dfn}
A \textit{weight-graded cofibrant cdga} is the colimit of free weight-graded extensions $A_{i}=A_{i-1}\otimes_d\Lambda V_{i}$,
starting from $A_0=\field$, where each $V_i$ is a bigraded vector space where the first grading is concedntrated in degree $i$. It is said to be \textit{minimal} if such 
extensions are ordered by non-decreasing positive cohomological degrees. 
If $A$ is a cdga with a weight decomposition, then a 
\textit{weight-graded cofibrant (resp. minimal) model} of $A$  is given by a weight preserving quasi-isomorphism $M\to A$ where $M$ is a weight-graded cofibrant (resp. minimal) cdga. 
\end{dfn}

The original construction of minimal models for connected cdga's extends to the weight-graded setting without surprises. In Lemma 3.2 in \cite{bcc} it is shown that positive weights are preserved under the minimal model construction. In fact, the same proof gives the following: 

\begin{lemma}\label{cdgamodelpositiveweights}
Let $A$ be a connected cdga with a weight decomposition whose induced decomposition in cohomology is positive. Then there is a weight-graded minimal model of $A$ with a positive weight decomposition. This is unique up to weight-graded isomorphism.
\end{lemma}

A dgla $L$ is said to be \textit{connected} if it is concentrated in non-positive cohomological degrees. 
Minimal dgla models are defined and constructed analogously: in the simply connected case (dgla's concentrated in strictly negative cohomological degrees) minimal dgla models exist (see e.g. \cite[§ 24]{felixrht}). Neisendorfer \cite{neisendorfer} conjectured that connected dgla's with nilpotent cohomology also admit minimal models. This is to the knowledge of the authors not proved, but a completed version of the statement holds (see e.g. \cite[Proposition 3.16]{bfmt}). However, quasi-free (not necessary minimal) dgla models for such dgla's always exist and are given by $\mathscr L^\star(\mathcal{C}^*_{CE}(L))$.
If $\Aa_{CE}^*(L)$ is formal, then a minimal dgla model for $L$ is given by $\mathscr L(H^*(\Aa_{CE}^*(L)))$. By Proposition \ref{prop:koszul-duality}, these models promote to the weight-graded setting. Moreover, we have:

\begin{prop}\label{prop:weights-on-indec-new}
Let $\Lambda V$ be a connected weight-graded cofibrant cdga and $\L W$ 
be a connected weight-graded cofibrant dgla.
\begin{enumerate}[label=(\alph*)]
\item\label{item:indecone}  There is an isomorphism of bigraded vector spaces
\[H^*(V)^\vee \cong s^{-1} H^*(\mathscr L(\Lambda V)).\]
\item\label{item:indectwo}  There is an isomorphism of bigraded vector spaces
\[H^*(W) \cong s H^*(\overline{\mathcal{C}}_{CE}^*(\L W)).\]
\end{enumerate}
\end{prop}
\begin{proof}
By Proposition \ref{prop:koszul-duality} 
the quasi-isomorphism $q_a\colon \Aa^*_{CE}(\mathscr L(\Lambda V))\to \Lambda V$ preserves weights. The induced map on the indecomposables $s\mathscr L(\Lambda V)^\vee \to V$ is a quasi-isomorphism (see \cite[§ 12.1.3]{lodayvallette}) which is also weight preserving. Dualizing this quasi-isomorphism we obtain the first statement.  
The second statement is proved similarly.
\end{proof}

\subsection{Weight-graded homotopy transfer}
Let $\Oo$ be one of the operads $\mathcal Ass$, $\Com$ or $\Lie$. Recall that for these specific operads, an $\Oo_\infty$-algebra structure on a differential graded vector space $A$ is completely determined by a collection of maps $\mu_i\colon A^{\ot i}\to A$, with $i\geq 1$, where $\mu_1$ is the differential of $A$ and the collection $\{\mu_i\}$ satisfies certain compatibility conditions depending on the operad $\Oo$. 
When considering weight-graded $\Oo_\infty$-algebras such operations are assumed to preserve weights.
There are obvious notions of 
weight-graded $\infty$-$\Oo_\infty$-morphism
and weight-graded minimal $\Oo_\infty$-models.

The following result appears in \cite{BudurRubio} in the Lie setting. The proof is an adaptation to the weight-graded setting of the classical proof and works more generally for other operadic algebras.
\begin{prop}\label{prop:weightgradedHTT}
Let $\Oo$ be one of the operads $\mathcal Ass$, $\Com$ or $\Lie$.
\begin{enumerate}[label=(\alph*)]
    \item\label{item:HTT-minimal-model} 
    Weight-graded minimal $\Oo_\infty$-models exist and are unique up to weight preserving $\infty$-$\Oo_\infty$-isomorphisms.
    \item\label{item:HTT-inverse} The inverse of a weight preserving $\infty$-$\Oo_\infty$-isomorphism is also weight preserving.
    \item\label{item:HTT-htpy-inverse} A weight preserving $\infty$-$\Oo_\infty$-quasi-isomorphism has a weight preserving homotopy inverse.
\end{enumerate}
\end{prop}
\begin{proof}
We sketch the idea of the proof for completeness. Let $(A,\{\mu_k\})$ be a weight-graded $\Oo_\infty$-algebra and $B$ a weight-graded dg vector space. Assume there is a homotopy retract
    $$
    \xymatrix{A\ar@(dl,ul)^{h}\ar@<0.5ex>[r]^-p &B\ar@<0.5ex>[l]^-i}
    $$
where $p$, $i$ and $h$ are weight preserving and $p$ and $i$ are cochain maps. Then there is an explicit formula for a transferred $\Oo_\infty$-algebra structure $(B,\{\eta_k\})$ where each $\eta_k$ is given as linear combinations of maps given as compositions of $p$, $i$, $h$, and $\mu_\ell$, where $\ell\leq k$ (we refer the reader familiar with the theory of operads to \cite[§ 10.3]{lodayvallette}). Since all of these maps are weight preserving, it follows that $\eta_k$ is also weight preserving.

One also shows that $i$ extends to an  $\infty$-$\Oo_\infty$-quasi-isomorphism 
$$i^\infty\colon (B,\{\eta_i\})\rightsquigarrow (A,\{\mu_i\}),$$ 
whose explicit formula is again given by linear combinations of maps given as compositions of the maps mentioned earlier (see \cite[§ 10.3.5]{lodayvallette}). Since there exists a weight preserving homotopy retract in which $B = H^*(A)$ the existence part of statement \textit{\ref{item:HTT-minimal-model}} follows.

For the uniqueness part, we need some more: If $B =H^*(A)$, then one can show that also $p$  extends to an $\infty$-$C_\infty$-quasi-isomorphism $p^\infty$, which by similar arguments as above is weight preserving.

Statement \textit{\ref{item:HTT-inverse}} is proved again by inspecting the explicit formula for the inverse of a weight preserving $\infty$-isomorphism $\{f_k\}\colon A\rightsquigarrow B$ (which again is a linear combination of compositions the maps that are weight preserving).
Now given all this, the proof of \cite[Theorem 10.4.4]{lodayvallette} holds word by word in this setting as well, yielding the proof of \textit{\ref{item:HTT-htpy-inverse}}.

The uniqueness part of \textit{\ref{item:HTT-minimal-model}} follows, since any two minimal weight-graded models will have a weight-graded $\infty$-quasi-isomorphism between them, which is a weight-graded $\infty$-isomorphism by minimality.
\end{proof}

\subsection{Algebraic models of spaces}
The rational homotopy type of any nilpotent space $X$ of finite type
is algebraically modelled by either cdga's (as in Sullivan's approach) or by dgla's (as in Quillen's approach).
The Sullivan minimal model of a connected topological space $X$, defined uniquely up to isomorphism, is given by a minimal cdga cofibrant model of the algebra of piece-wise linear forms $\Aa_{pl}(X)$ on $X$. 
On the other hand, if $X$ is simply connected, then the Quillen minimal model of $X$ is given by a minimal dgla model
of the Lie algebra $\lambda(X)$, where $\lambda$ denotes Quillen's construction in \cite{quillenRHT}.
Note that $\lambda$ is only defined for simply connected spaces, but the theory of dgla models can be extended to nilpotent spaces of finite type by
taking first a Sullivan minimal model and then applying Quillen's functor 
$\mathscr L$ from finite type cdga's to dgla's (see \cite{neisendorfer}). A further extension of Quillen models to general topological spaces (not necessarily connected or nilpotent) is treated in \cite{bfmt}.
Recall that a topological space is said to be \textit{formal} if it has a formal Sullivan model. It is called \textit{coformal} if it has a formal Quillen model.

By Proposition \ref{prop:koszul-duality} and Proposition \ref{prop:weightgradedHTT},
a weight graded algebraic model for $X$ (cdga, $C_\infty$-algebra, dgla, $L_\infty$-algebra) is uniquely represented, up to weight-graded isomorphism in the corresponding category, by either a weight-graded cdga minimal model, a weight graded minimal $C_\infty$-algebra model or a weight graded minimal $L_\infty$-algebra model.

In the simply connected case, a weight graded algebraic model is also uniquely represented by a unique weight graded minimal dgla model. In the non-simply connected case, it is not known whether minimal dgla models always exist (Neisenforfer cojectured their existence, and a proof of a completed version of the Neisendorfer conjecture can be found in \cite[Proposition 3.16]{bfmt}). In Section \ref{existenceminimalmodels} we discuss properties that guarantee the existence of such models.

\begin{rmk}\label{rmk:induced-weights-on-indec}
A non-trivial  weight decomposition on a Sullivan minimal model $\Lambda V$ for a nilpotent space $X$ of finite type is a source of weight decompositions on the rational homotopy groups in two different ways. First, there is an induced weight decomposition on $H^*(V)\cong \pi_*^\Q(X)^\vee$. Second, there is an induced weight decomposition on 
\[H^*(\mathscr L(\Lambda V))\cong \pi_{-*}^\Q(\Omega X) = \pi_{-*-1}^\Q(X).\] By Proposition \ref{prop:weights-on-indec-new} \textit{\ref{item:indecone}} these two weight decompositions coincide.

Similarly, a non-trivial indecomposable induced weight decomposition on a dgla model for $X$ is a source of two weight decompositions on the rational (co)homology of $X$, which coincide by Proposition \ref{prop:weights-on-indec-new} \textit{\ref{item:indectwo}}.
\end{rmk}

\section{Restrictions on weights}
In this section we consider various properties that restrict the range of weights in our algebras. We first introduce the formality and coformality induced weights. Then review the theory of positive and pure weights and introduce an intermediate property which gives vanishing of higher operations in cohomology.

\subsection{Formality and coformality induced weights}
If a nilpotent space $X$ of finite type is formal, then of course it has a weight-graded model, namely its cohomology $H^*(X)$ with weights given by $w(x)=|x|$. The theory of weight-graded minimal models gives a unique up to isomorphism weight-graded minimal cdga model $\Lambda V\to H^*(X)$ and a unique  weight-graded dgla model $\mathscr L(H^*(X))$. We call these the \textit{formality induced weight decompositions} on the algebraic models for the space $X$.

\begin{exmp}\label{exmp:lie-model-for-cpn} Since $\CP^k$ is formal, its cohomology is its own cdga model: the algebra $\Q[u]/(u^{k+1})$, $|u|=2$, $w(u)=2$ with the trivial differential is a weight-graded cdga model for $\CP^k$. The formality induced weight decomposition on the dgla model of $\CP^k$ is explicitly given by 
$$
\L(v_1,\dots, v_k),\, |v_i| =1-2i,\, w(v_i) = -w(u^i)=-2i,\, d(v_i) = \frac12\sum_{a+b=i}[v_a,v_b].
$$
\end{exmp}

If a nilpotent space $X$ of finite type is coformal, then of course it has a weight-graded Lie model, namely the rational homotopy groups $\pi_*(\Omega X)\ot \Q$ on the based loop space of $X$  with weights given by $w(x)=|x|$.  If $X$ is simply connected, there is a unique up to isomorphism weight-graded minimal  dgla model  $\L W\to \pi_*(\Omega X)\ot \Q$ and a unique  weight-graded Sullivan minimal model $\Aa^*_{CE}(\pi_*(\Omega X)\ot \Q ) $. We call these the \textit{coformality induced weight decompositions} on the algebraic models for $X$. If $X$ is not simply connected it is still possible to deduce the existence of minimal dgla models, for instance when $X$ is formal (other cases discussed in Section \ref{existenceminimalmodels}).

We state two results about formality and coformality induced weight decompositions that we will use later on.

\begin{prop}\label{prop:weight<deg} 
Let $X$ be a nilpotent space of finite type.
\begin{enumerate}[label=(\alph*)]
    \item\label{item:formal} If $X$ is formal, then the formality induced weight decomposition on its Quillen minimal model $\L W$ satisfies $w(x)<|x|$ for every non-trivial $x\in \L W$ of pure weight $w(x)$ and cohomological degree $|x|$. 
    \item\label{item:coformal} If $X$ is coformal, then the coformality induced weight decomposition on its Sullivan minimal model $\Lambda V$ satisfies $w(x)<|x|$ for every non-trivial $x\in \Lambda V$ of pure weight $w(x)$ and cohomological degree $|x|$.
\end{enumerate}
\end{prop}

\begin{proof}
Given the formality induced weight decomposition on $\L W$, there is an isomorphism of weight-graded vector spaces $W \cong s(H^*(X;\Q))^\vee$ (see Remark \ref{rmk:induced-weights-on-indec}). In particular, for every $x\in W$ we have that $w(x) = |x|-1<|x|$. The proof of the second statement follows similarly.
\end{proof}

\begin{rmk}
  The above inequalities  hold for the cohomologies $H^*(\L W) = \pi_*(\Omega X)\ot \Q$ and $H^*(\Lambda V)= H^*(X;\Q)$ and thus for the minimal $L_\infty$- and $C_\infty$-models in respective case.  
\end{rmk}

\subsection{Weights on a tilted segments and purity}
We now consider weights restricted to a certain segment, generalizing the purity context. We prove that segmented weights give homotopical restrictions.

We will again assume that $\Oo$ is one of the operads $\mathcal Ass$, $\Com$ or $\Lie$. 

\begin{dfn}
Let $\alpha>0$ be a rational number and $k\geq 0$ an integer. We say that a weight-graded dg $\Oo$-algebra $A$ is \textit{$(\alpha,k)$-segmented} if $H^n(A)_p$ is non-trivial only when \[\alpha n\leq p\leq \alpha(n+k).\]
\end{dfn}
\begin{exmp}
If $A$ is $(1,2)$-segmented then the weights in the cohomology $H^*(A)$ are concentrated in the segment visualized below:
\begin{center}
    \includegraphics[scale=0.6]{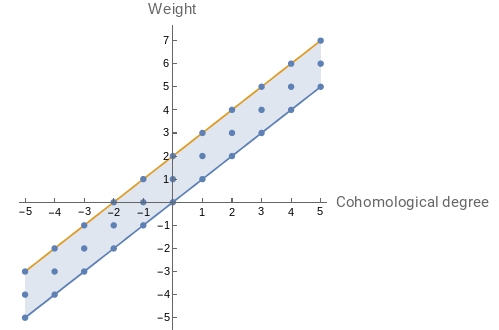}
\end{center}
\end{exmp}

\begin{rmk}
 Being $(\alpha,0)$-segmented is equivalent to being $\alpha$-pure in the sense of \cite{CiHo}. If $\alpha=a/b$ with $a$ and $b$ coprime, this implies that 
 $H^*(A)$ is concentrated in degrees 
 that are divisible by $b$, and $H^{kb}(A)$ is pure of weight $ka$.
  In particular, the case of $(1,0)$-segmented is the classical purity one encounters for compact Kähler manifolds.
Note that $(\alpha,k)$-segmented implies $(\alpha,k')$-segmented for every $k'\geq k$.
\end{rmk}

We will say that an $\Oo_\infty$-algebra $(A,\mu_1,\mu_2,\dots)$  has a minimal $\Oo_\infty^{\leq k}$-model, if it has a minimal $\Oo_\infty$-model
with vanishing operations $\mu_m=0$ for $m>k$.
Note that having a minimal $\Oo_\infty^{\leq 2}$-model is equivalent to formality
of the $\Oo_\infty$-algebra. We have:

\begin{prop}\label{prop:semipurityImpliesSemiformality}
Every weight-graded $(\alpha,k)$-segmented dg $\Oo$-algebra  has a minimal $\Oo_\infty^{\leq k+2}$-model.
\end{prop}
\begin{proof}
Let $(H,0,\mu_2,\mu_3,\dots)$ be a weight-graded minimal $\Oo_\infty$-model for $A$. Let $m>k+2$ and assume to get a contradiction that there are elements $x_1,\dots,x_m\in H$ of homogeneous weights and cohomological degrees such that $\mu_m(x_1\ot\cdots\ot x_m)\neq 0$.
By taking into account that 
\[|\mu_m(x_1\ot\dots\ot x_m)| = 2-m+\sum |x_i|,\] it follows that the weight of $\mu_m(x_1\ot\dots\ot x_m)$, satisfies
$$w(\mu_m(x_1,\dots,x_m))\leq \alpha (k+2-m+\sum |x_i|)<\alpha\sum|x_i|.$$

Since $\mu_m$ is weight preserving it follows that that 
$$w(\mu_m(x_1\ot\dots\ot x_n))\geq \alpha\sum|x_i|.$$

The two inequalities above contradict each other, yielding that $\mu_m=0$.
\end{proof}

\begin{rmk}
A similar argument as in the proof of Proposition \ref{prop:semipurityImpliesSemiformality} gives vanishing of Massey products of length $>k+2$ for any dga with an $(\alpha,k)$-segmented weight decomposition.
\end{rmk}

\begin{exmp}
On 
the cohomology of a $d$-dimensional smooth complex algebraic variety, weights arising from mixed Hodge theory are concentrated in the triangle given by the following: Let $p$ denote the weight and $n$ the cohomological degree. Then the range of weights is given by the following inequalities; $p\leq d$ and $n\leq p\leq 2n$, as visualized below.
\begin{center}
    \includegraphics[scale=0.25]{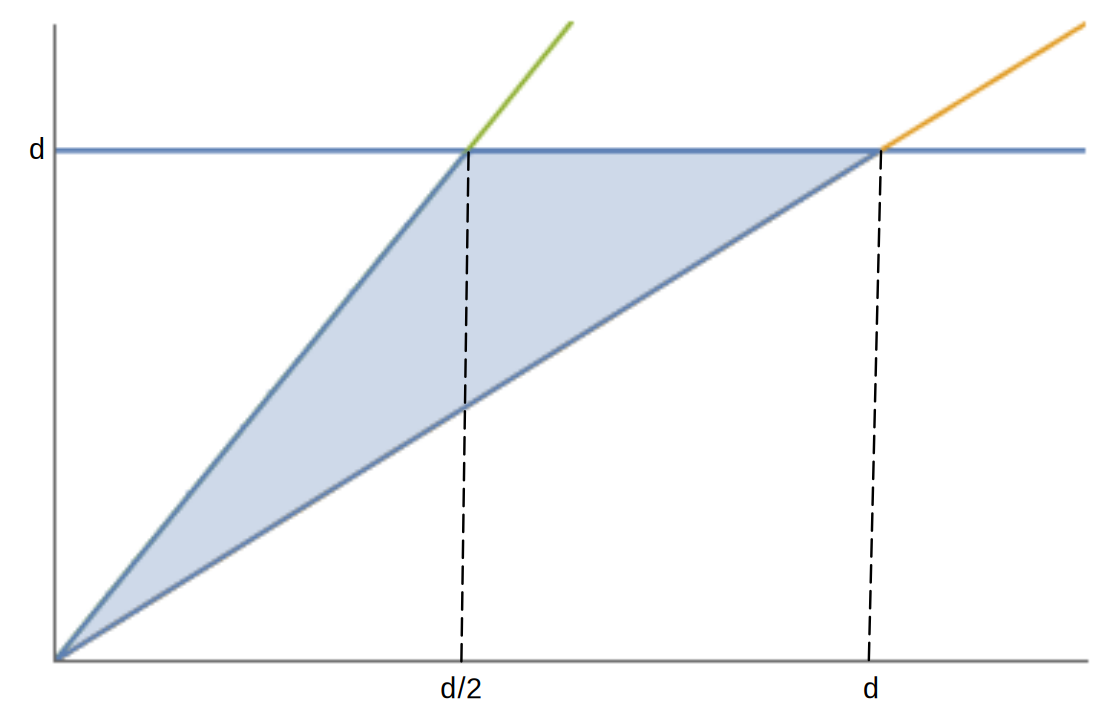}
\end{center}
We see that any $d$-dimensional smooth complex algebraic variety is $(1,d/2)$-segmented, yielding that the ordinary $n$-Massey products vanish for $n>d/2+2$.
\end{exmp}

\begin{exmp}In some situations, there are weight decompositions up to a certain degree, which are segmented up to this degree.
Then the vanishing of higher operations is still true in the corresponding range of degrees.
This is the case, for instance, of any
compact complex manifold admitting a transverse Kähler structure on a fundamental central foliation: by \cite{Kasuya} such a manifold $M$ admits a model which, in degrees $\leq 2$ has a weight decomposition where
$H^1(M)$ is concentrated in weights 1 and 2, and $H^2(M)$ has weights 2, 3 and 4. As a consequence, one infers that
there are no $k$-tuple Massey products of elements in $H^1(M)$ for $k\geq 5$.
\end{exmp}

\subsection{Positive weights}
Positive weights together with finite-dimensional cohomology, lead to segmented weights and so give vanishing results of certain higher operations. 
First, note that there are various equivalent characterizations for a space to have positive weights:

\begin{prop}\label{prop:postive-negative-weights-new} 
Given a nilpotent space $X$ of finite type the following are equivalent:
\begin{enumerate}[label=(\alph*)]
    \item\label{aa} $X$ has a cdga model  with a positive weight decomposition.
    \item\label{bb} $X$ has a cdga model with a weight decomposition such that the weights are positive on cohomology.
    \item\label{cc} $X$ has a dgla model  with a negative weight decomposition.
    \item\label{dd} $X$ has a dgla model with a weight decomposition such that the weights are negative on cohomology.
\end{enumerate}
\end{prop}

\begin{proof} That \textit{\ref{aa}} and \textit{\ref{bb}} are equivalent follows from Lemma \ref{cdgamodelpositiveweights}.
The implication \textit{\ref{aa}} $\Rightarrow$ \textit{\ref{cc}} follows by applying $\mathscr L$ to $A$ if $A$ is of finite type. If $A$ is not of finite type, we might apply $\mathscr L$ to a weight graded minimal Sullivan model of it, again using Lemma \ref{cdgamodelpositiveweights}.

The implication \textit{\ref{cc}} $\Rightarrow$ \textit{\ref{aa}} follows by applying $\Aa^*_{CE}$ to $L$ if it is of finite type. If $L$ is not of finite type, we might apply $\Aa^*_{CE}$ to a minimal $L_\infty$-algebra model of it; such models exist have positive weights by  Proposition \ref{prop:weightgradedHTT}.  Hence \textit{\ref{aa}} $\Leftrightarrow$ \textit{\ref{cc}}.

The implication \textit{\ref{cc}} $\Rightarrow$ \textit{\ref{dd}} is clear. For  \textit{\ref{dd}} $\Rightarrow$ \textit{\ref{cc}}, let $L$ be a dgla model for $X$ where $H^*(L)$ has negative weight decomposition. By Proposition \ref{prop:weightgradedHTT} it follows  that $L$ has a negatively graded minimal $L_\infty$-algebra model $(H^*(L),\{\ell_i\})$. Now the dgla $\mathscr L(\Aa_{CE}^*(H^*(L),\{\ell_i\}))$ is a dgla model for $X$ which is negatively graded.
\end{proof}

In view of the above result, we define:
\begin{dfn}
A topological space is said to have \textit{positive weights} if it admits a Sullivan model with a positive weight decomposition or, equivalently, a Quillen model with a negative weight decomposition.
\end{dfn}

The following is straightforward:
\begin{prop}\label{thm:positive+ftdim}
Let $A$ be a dga  with a positive weight decomposition. If $H^*(A)$ is finite dimensional, then 
there is some $\alpha>0$ and $k\geq 0$, so that $A$ is $(\alpha,k)$-segmented. 
\end{prop}

As a consequence, $k$-tuple Massey products on the cohomology of topological spaces with positive weights 
vanish for sufficiently large $k$.

\section{Applications}
In this last section, we prove our main results on mapping spaces, classifying spaces of homotopy automorphisms
and free and based loop spaces. We also show that nilpotent spaces 
with positive weights always admit minimal dgla models.

\subsection{Mapping spaces} Given a connected space $X$, and a nilpotent space $Y$ of finite $\Q$-type, Berglund constructs in \cite{berglund15} explicit $L_\infty$-algebra models for the different connected components of $\map(X,Y)$.  We show that when $X$ and $Y$ carry weight decompositions, it is possible to equip the Berglund models with a weight decomposition, under certain conditions. This applies, for instance, to components of algebraic maps between complex algebraic varieties or holomorphic maps between compact Kähler manifolds.

We first recall the main constructions treated in \cite{berglund15}.

Given an $L_\infty$-algebra $(L,\{\ell_i\})$ let $\Gamma^kL$ be the span of all possible
elements of $L$ one can form using at least $k$ elements from $L$ and the maps $\ell_1,\ell_2,\dots$ (for instance, $\ell_3(x,\ell_2(y,z),\ell_1(t))$ belongs to $\Gamma^4L$).
This gives a filtration 
$$
L =\Gamma^0L\supseteq \Gamma^1L\supseteq \cdots.
$$
The $L_\infty$-algebra is called \textit{degree-wise nilpotent} if for every cohomological degree $n$ there is some $k$ such that $(\Gamma^k L)^n=0$. Minimal $L_\infty$-models for nilpotent spaces of finite $\Q$-type turn out to be degree-wise nilpotent (\cite[Theorem 2.3]{berglund15}).

The Maurer-Cartan elements of a degree-wise nilpotent $L_\infty$-algebra $(L,\{\ell_i\})$ are given by
\[
\MC(L):=\{\tau\in L^1\ |\ \sum \ell_i(\tau^{\ot i})/i!=0\}.
\]
Note that the sum $\sum \ell_i(t^{\ot i})/i!$ converges due to degree-wise nilpotency.

Given a cdga $A$ and an $L_\infty$-algebra $(L,\{\ell_i\})$, the tensor product $A\ot L$ is an $L_\infty$-algebra $(A\ot L,\{\ell^A_i\})$ where 
$$\ell_i^A((a_1\ot x_1)\ot\cdots\ot(a_i\ot x_i))= \pm a_1\cdots a_i\ot \ell_i(x_1\ot\cdots x_i).$$
If $A$ is connected and bounded and $(L,\{\ell_i\})$ is degree-wise nilpotent, then $(A\ot L,\{\ell_i^A\})$ is again a degree-wise nilpotent $L_\infty$-algebra and it makes sense to consider the set of Maurer-Cartan elements $\MC(A\ot L)$. 

Assume now that $L$ is of finite type. Given an element $\tau\in \MC(A\ot L)$, there is an associated morphism
\[\varphi_\tau\in \Hom_{\mathrm{cdga}}(\Aa^*_{CE}(L),A)\]
defined as follows: Since $L$ is of finite type, we have that the underlying graded algebra structure of  $\Aa^*_{CE}(L) $ is given by the free graded commutative algebra $\Lambda(s^{-1} L)^\vee$. Then for $\xi\in (s^{-1}L)^\vee$ we set \[\varphi_\tau(\xi) :=(1\ot\xi)(\tau).\]

If either $A$ or $L$ is finite dimensional, Berglund \cite[Proposition 6.1]{berglund15} shows  that there is an isomorphism of sets  
\[
\MC(A\ot L)\stackrel{\cong}{\longrightarrow} \Hom_{\mathrm{cdga}}(\Aa^*_{CE}(L),A)
;\quad
\tau\mapsto\varphi_\tau.
\]
Moreover, the equivalence respects homotopy: two elements in $\MC(A\ot L)$ are gauge equivalent if and only if their corresponding maps in $\Hom_{\mathrm{cdga}}(\Aa^*_{CE}(L),A)$ are homotopic (in the sense of \cite[§ 12 (b)]{felixrht}).

\begin{dfn}
Given a weight-graded dgla $L$, let $W_0\MC(L)$ denote the set of weight-zero Maurer-Cartan elements:
 $$W_0\MC(L) := \{\tau\in\MC(L)\ |\ w(\tau)=0\} = \MC(L_0).$$
\end{dfn}

\begin{lemma}
Let $A$ be a weight-graded connected cdga and let $L$ be a weight-graded connected dgla of finite type. Let 
$\tau\in \MC(A\otimes L)$. Then $\tau \in W_0\MC(A\ot L)$ if and only if $\varphi_\tau$ is weight preserving. 
\end{lemma}
\begin{proof}
It is enough to prove that $\varphi_\tau$ preserves the weight on an element in the indecomposables $\xi\in (s^{-1}L)^\vee$. Assume $\xi:s^{-1}L\to \Q$ is of weight $n$. Then it increases the weight by $n$ (in this case, it vanishes on elements of weights $w\neq -n$). For a weight zero Maurer-Cartan element 
\[\tau = \sum a_i\ot x_i\in W_0\MC(A\ot L)\] we must have that $w(a_i)+w(x_i)=0$ and hence $w(a_i\ot \xi(s^{-1}x_i))= w(\xi)$. In particular
$$w(\varphi_\tau(\xi))=w(1\ot\xi s^{-1}(\tau)) = w\left(\sum a_i\ot \xi(s^{-1}x_i)\right) =w(\xi).  $$
Hence $\varphi_\tau$ is weight preserving.
\end{proof}

For the following result, we will require two standard constructions on dgla's:

\begin{dfn}
Given a degree-wise nilpotent $L_\infty$-algebra $(L,\{\ell_i\})$ and a Maurer-Cartan element $\tau\in\MC(L)$, let $(L^\tau,\ell_i^\tau)$ denote the $\tau$-twisted $L_\infty$-algebra 
where
$$
\ell_i^\tau(x_1\ot\cdots\ot x_i) = \sum_{k\geq0}\frac1{k!}\ell_{k+i}(\tau^{\ot k}\ot x_1\ot\cdots\ot x_i).
$$
\end{dfn}

\begin{dfn}
The \textit{$n$-connected cover} of a an $L_\infty$ algebra $(L,\{\ell_i\})$, denoted by  $(L\langle n\rangle,\{\ell_i\})$, has an underlying graded vector space given by
 is given by
$$L\langle n \rangle^i=\left\{
\begin{array}{ll}
L^i&\quad\text{if }i<n,\\
\mathrm{ker}(L^n\xrightarrow d L^{n+1})&\quad\text{if }i=n,\\
0&\quad\text{if }i>n.
\end{array}
\right.$$
The structure maps are given by restrictions.
\end{dfn}

We have all the preliminaries to state a weight-graded version of \cite[Theorem 1.5]{berglund15}.

\begin{thm}\label{thm:weights-on-mapping-space}
Let $X$ be a connected space with a weight-graded cdga model $A$ and let $Y$ be a connected nilpotent space of finite $\Q$-type with a weight-graded degree-wise nilpotent $L_\infty$-algebra model $(L,\{\ell_i\})$, where either $A$ or $L$ is finite dimensional.   Given a map $f\colon X\to Y$ whose rational homotopy class corresponds to some  $\tau\in W_0\MC(A\ot L)$, then the $L_\infty$-algebra model given by $(A\ot L)^\tau\langle 0\rangle$ is a weight-graded $L_\infty$-algebra model for $\map^f(X,Y)$.
\end{thm} 
\begin{proof}
If $\g$ is a weight-graded $L_\infty$-algebra and $\tau\in W_0\MC(\g)$, then $\g^\tau$ is also a weight-graded $L_\infty$-algebra. The rest follows from  \cite[§ 6]{berglund15}.
\end{proof}

\begin{rmk}As shown in Example \ref{examvarsE1} below, algebraic varieties always have finite-dimensional models and so the above theorem applies.
An alternative description for algebraic models of mapping spaces, where finite-dimensionality of the initial algebraic models is not necessary, is given in Proposition 1.9 of \cite{CPRW}. This description might be convenient in order to study the whole mixed Hodge structure on mapping spaces of algebraic varieties.
\end{rmk}

We review two main examples where the above theorem applies:

\begin{exmp}\label{examkahlermaps}
Let $f:X\to Y$ be a formal map between topological spaces of finite type. For instance, this situation applies to any holomorphic map between compact Kähler manifolds.
The map \[f^*:A(Y):=H^*(Y)\to A(X):=H^*(X)\] is a finite-dimensional model of $f$ which is weight-graded by letting \[A^n_n(X)=H^n(X)\text{ and }A^n_p=0\text{ otherwise}.\]
Let $L_Y$ denote the minimal $L_\infty$-algebra model with the formality induced weight decomposition. Since $L_Y$ is degree-wise nilpotent, it follows that $\Aa_{CE}^*(L_Y)$ is a minimal model of $A(Y)$ as weight graded algebras. We obtain a weight preserving morphism \[\varphi:\Aa^*_{CE}(L_Y)\to A(X).\]
The Maurer-Cartan element $\tau_\varphi\in \MC(H^*(X)\ot L_Y)$ corresponding to $\varphi$ will then have weight zero.
\end{exmp}

\begin{exmp}\label{examvarsE1}Let $f:X\to Y$ be a morphism of connected complex algebraic varieties. Such a map is modelled by the morphism of bigraded cdga's
\[f^*:(E_1(Y),d_1)\longrightarrow (E_1(X),d_1)\] given by the multiplicative weight spectral sequence (see \cite{CG1}). It is a finite-dimensional model and moreover admits a weight decomposition by rearranging the bidegrees:
let \[A^n_p(X):=E_1^{n-p,p}(X).\] Then the map
$f^*:A(Y)\to A(X)$ is a finite-dimensional weight-graded cdga model of $f$. Furthermore, in the case of smooth algebraic varieties (not necessarily projective) the weight decompositions are positive.
Arguing as in the previous example, we obtain a weight preserving morphism \[\varphi:\Aa^*_{CE}(L_Y)\to A(X).\]
The Maurer-Cartan element corresponding to $\varphi$ will then have weight zero. 
Note that in the case of a smooth projective variety $X$ we recover the weight decomposition of the previous example.
\end{exmp}

\begin{exmp}
 Let $e\colon \CP^k\hookrightarrow \CP^{k+n}$, $n\geq 1$, be the standard embedding which is modelled by the projection \[\Q[u]/u^{k+n+1}\to \Q[u]/u^{k+1},\] which is a weight preserving morphism.
 Hence $\map^e(\CP^k, \CP^{k+n})$ has weight graded algebraic models by Theorem \ref{thm:weights-on-mapping-space} where the weights are inherited from the formality induced weights on the algebraic models for $\CP^k$ and $\CP^{k+n}$. We will show that $\map^e(\CP^k,\CP^{k+n})$ has a minimal $\mathcal Com_\infty^{\leq k+2}$-model by endowing a Sullivan minimal model, computed in \cite[§ 7]{berglund15}, with weights.
A minimal model for $\map^e(\CP^k,\CP^{k+n})$ is given
\[A=\Lambda(z,y_n,y_{n+1},\dots, y_{n+k}),\quad |z|=2,\ |y_r|=2r+1,\quad dz=0,\ dy_r =z^{r+1}\]
(see \cite{berglund15} for further details).
By explicit computations using the formality induced weights one gets that
$$w(z) = 2,\ w(y_r)=2r+2.$$

We will prove that $A$ is $(1,k)$-segmented, and therefore has a minimal $\mathcal Com_\infty^{\leq k+2}$-model by Proposition \ref{prop:semipurityImpliesSemiformality}. 
First we observe that $w(z)= |z|$ and that $w(y_j) = |y_j|+1$. Hence 
$$w(z^a y_{i_1}\cdots y_{i_p}) = |z^a y_{i_1}\cdots y_{i_p}| + p.$$
Since $y_j^2= 0$, it follows that if $p>k+1$ then $z^a y_{i_1}\cdots y_{i_p}=0$. From this we may deduce $(1,k+1)$-segmentation. However, we observe that any element of cohomological degree $\ell$ and of weight $\ell +k+1$ is a multiple of $z^ay_n\cdots y_{n+k}$ (where $a=\ell-\sum_{i=n}^{n+k}(2i+1)$) and those elements are non-cycles. Hence there is no cohomology class of cohomological degree $\ell$ and of weight $\ell+k+1$. From this we conclude the $(1,k)$-segmentation.
\end{exmp}

\begin{rmk}
Note that the deduction above is not sharp: weights imply that there is a $\Com_\infty^{\leq k+2}$-model for $\map^e(\CP^k, \CP^{k+n})$ but, in fact, $\map^e(\CP^k, \CP^{k+n})$ is formal. Indeed,
consider the cdga $B=\Lambda(z,y_n,w_{n+1},\dots,w_{n+k})$ where $z$ and $y_n$ is as before but with $|w_r|= 2r+1$ and $dw_r = 0$. Define a morphism $\varphi\colon A\to B$ where
\[\varphi(z)=z,\, \varphi(y_n)=y_n\text{ and }\varphi(y_{n+i}) = w_{n+i}+z^{i}y_n.\]
An inverse $\psi\colon B\to A$ is given by
\[\psi(z)=z,\,\psi(y_n)=y_n\text{ and }\psi(w_{n+i})= y_{n+i}-z^iy_n.\] In particular $B$ is a cdga model for $\map^e(\CP^k,\CP^{k+n})$. Moreover we can decompose $B$ as the tensor product
$\Lambda(z,y_n)\ot \Lambda(w_{n+1})\ot\cdots\Lambda(w_{n+k})$ which is a  model for the product
\[\CP^n\times S^{2n+3}\times S^{2n+5}\times\cdots\times S^{2(k+n)+1}.\]
\end{rmk}

In the next two results we provide conditions that guarantee positive weights for certain components of mapping spaces.

\begin{prop}\label{posweightsgeneral} Let $X$ be a connected space and $Y$ be a nilpotent space of finite $\Q$-type. Let $\Lambda V$
 be a minimal model for $X$ and let $L_Y$ be a minimal model $L_\infty$-algebra model for $Y.
 $\begin{enumerate}[label=(\alph*)]
    \item\label{item:formal-positive-weights} If $X$ and $Y$ are formal and $\dim(H^*(X;\Q))<\infty$ and  $f\colon X\to Y$ corresponds to some  weight-zero Maurer-Cartan element $\tau\in W_0\MC(H^*(X;\Q)\ot L_Y)$ with respect to the formality induced weights, then  $\map^{f}(X,Y)$ has positive weights. 
    \item\label{item:coformal-positive-weights} If $X$ and $Y$ are  coformal and $\dim(\pi_*^\Q(Y))<\infty$ and $f\colon X\to Y$ corresponds to some  weight-zero Maurer-Cartan element $\tau\in W_0\MC(\Lambda V\ot \pi_*^\Q(Y))$ with respect to the coformality induced weights, then $\map^{f}(X,Y)$ has positive weights. 
\end{enumerate}
\end{prop}

\begin{proof} \textit{\ref{item:formal-positive-weights}}
Let $a\ot b\in H^*(X)\ot L_Y$ be an element of non-positive cohomological degree, so that $|a|+|b|\leq0$. Since $X$ and $Y$ are formal we have $w(a)= |a|$ and $w(b)<|b|$ by Proposition \ref{prop:weight<deg} \textit{\ref{item:formal}}. In particular $w(a\ot b) = w(a)+w(b)<0$. From this we conclude that $(H^*(X;\Q)\ot L_Y W)^\tau\langle 0\rangle$ is a dgla model  for $\map^f(X,Y)$ with negative weights, and thus $\map^f(X,Y)$ admits  a cdga model with positive weights by Proposition \ref{prop:postive-negative-weights-new}.
Assertion \textit{\ref{item:coformal-positive-weights}}  is proved similarly, using Proposition \ref{prop:weight<deg} \textit{\ref{item:coformal}} instead.
\end{proof}

Theorem \ref{second} in the introduction for holomorphic maps between Kähler manifolds now follows from Proposition \ref{posweightsgeneral} together with Example \ref{examkahlermaps}.

 We may now prove Theorem \ref{first} in the introduction.
Note that in order to get a non-trivial weight decomposition on the tensor product $A\ot L$, it is enough that one of the algebras has a non-trivial weight decomposition. Assuming that $L$ has a negative weight decomposition and $A$ has a trivial weight decomposition, we get that $(A\ot L)\langle 0\rangle$ has a negative weight decomposition. The only weight-zero Maurer-Cartan element is thus the trivial one, which corresponds to the constant map.

\begin{thm}
Let $X$ be of the homotopy type of a finite CW-complex and let $Y$ be a connected nilpotent space of finite $\Q$-type having positive weights. Let $\Psi\colon X\to Y$ denote a constant map. Then $\map^{\Psi}(X,Y)$ has positive weights. 
\end{thm}

\begin{proof}
Since $X$ is of the homotopy type of a finite CW-complex, it has a finite dimensional connected cdga model $A$, see for instance \cite[Example 6, p. 146]{felixrht}. Endow $A$ with the trivial weight decomposition. Since the Sullivan minimal model for $Y$ has a positive weight decomposition, it follows from Proposition \ref{prop:postive-negative-weights-new} that $Y$ has an $L_\infty$-algebra model $L$ with a negative weight decompositon. Therefore by Proposition \ref{posweightsgeneral}, $(A\ot L_Y)^\mathrm{triv}\langle0\rangle$, where $\mathrm{triv}$ is the trivial Maurer-Cartan element, is an $L_\infty$-algebra model for $\map^{\Psi}(X,Y)$ which admits a negative weight decomposition.
\end{proof}

\subsection{Classifying spaces of homotopy automorphisms}\label{sec:applII}
For a topological space $X$, let $\aut(X)$ ($\aut_*(X)$) denote the topological monoid of (pointed) homotopy automorphisms, i.e. the space of (pointed) endomorphisms $\varphi\colon X\to X$ that are homotopy equivalences. If $X$ is of the homotopy type of a finite CW-complex, then there is a universal $X$-fibration 
$$
X\to B\aut_*(X)\to B\aut(X).
$$

Given a map $f\colon B\to B\aut(X)$, we may apply the based loop space functor which up top homotopy yields a map $\Omega f\colon \Omega B\to \aut(X)$. If $B$ is simply connected, then $\Omega B$ is connected, and hence $\Omega f$ factors through the connected component of $\aut(X)$ that contains the identity. We denote this connected component by $\aut_\circ(X)$, which again is a topological monoid. Delooping this factorization gives that $f$ factors through $B\aut_\circ(X)$ if $B$ is simply connected. Hence $B\aut_\circ(X)$ classifies $X$-fibrations over simply connected spaces, via the $X$-fibration 
\[X\to B\aut_{*,\circ}(X)\to B\aut_{\circ}(X),\] where $\aut_{*,\circ}(X)$ denotes the connected component of the identity in $\aut_*(X)$.

Since $B\aut_\circ(X)$  and $B\aut_{*,\circ}(X)$ are simply connected and of finite type, they have rational cdga and dgla models. The dgla models are expressed in terms of derivations of certain algebras.
 
\begin{dfn}
Let $A$ be a cdga, then $\Der(A)$ is the \textit{dgla of graded derivations},
$$
\Der(A) = \{\theta\colon A\to A\ | \ \theta(a\cdot b) = \theta(a)\cdot b + (-1)^{|\theta||a|}a\cdot \theta(b)\},
$$
where the Lie bracket is given by 
$$
[\theta, \eta] = \theta\circ \eta - (-1)^{|\theta||\eta|}\eta\circ \theta
$$ and the differential $\partial\colon \Der(A)\to \Der(A)$ is given by $[d,-]$ where $d$ is the differential of $A$.

Similarly, one may define the \textit{dgla of derivations on a dgla $L$}. A derivation of a dgla $L$ is a linear map $\eta\colon L\to L$, such that \[\eta[a,b]= [\eta(a),b]+(-1)^{|\eta||a|}[a,\eta(b)].\]
The dgla of derivations on a dgla $L$ is denoted by $\Der(L)$.
\end{dfn}

\begin{lemma}\label{lemma:weight-on-derivations} Let $\Lambda V$ (resp. $\L W$) be a weight-graded cofibrant minimal cdga (resp. dgla). Then there  is an induced weight grading on the dgla $\Der(\Lambda V)$ (resp. $\Der(\L W)$) so that \[\Der(\Lambda V)\cong \Hom(V,\Lambda V)\quad\text{(resp. }\Der(\L W)\cong \Hom(W,\L W)\text{ )}\] is an isomorphism of bigraded  vector spaces.
\end{lemma}
\begin{proof} We prove the statement for the derivations on the cofibrant minimal cdga. The statement regarding the derivations on the minimal dgla is proved similarly.

By definition, if $\Lambda V$ is a weight-graded minimal cdga, then $V$ is weight-graded. Hence the mapping space $\Hom(V,\Lambda V)$ has an induced weight grading; a linear map is of weight $w$ if it raises the weight by $w$. There is an isomorphism of vector spaces $\Hom(V,\Lambda V)\cong \Der(\Lambda V)$, which induces a weight grading on $\Der(\Lambda V)$. A derivation is of weight $n$ if it increases the weight by $n$.

The Lie bracket on $\Der(\Lambda V)$ preserves the weight.  
Since the differential on $\Lambda V$ preserves weights, the differential $\partial = [d,-]$  on $\Der(\Lambda V)$ is also weight preserving. 
\end{proof}

\begin{thm}\label{htpautpos}
Let $X$ be a simply connected space of the homotopy type of a finite CW-complex. If $X$ is formal or coformal, then  $B\aut_\circ(X)$ and $B\aut_{*,\circ}(X)$ have positive weights.
\end{thm}

\begin{proof}
Let $\Lambda V$ and $\L W$ be Sullivan minimal and dgla models for $X$, respectively, both endowed with the (co)formality induced weight decomposition. It follows by the theory of Sullivan \cite{sullivan77}, Schlessinger and Stasheff \cite{SchlessingerStasheff} and Tanré \cite{tanre83}, that dgla models for $B\aut_\circ(X)$ and $B\aut_{*,\circ}(X)$ are given by $\Der(\Lambda V)\langle -1\rangle$ and $\Der(\L W)\langle -1\rangle$, respectively, which now have induced weight decompositions (Lemma \ref{lemma:weight-on-derivations}). By Proposition \ref{prop:postive-negative-weights-new}, it is enough to prove that these Lie models have a cohomology concentrated in negative weights in order complete the proof. 

Consider the filtration $0 = F_0 \subseteq F_1\subseteq\cdots$, where  $F_i$ is the space of derivations that vanish on generators of cohomological degree $>i$. In particular,
$$F_i \cong \Hom(V^{\leq i},\Lambda V)$$
is a weight-graded subcomplex of $\Der(\Lambda V)$. This gives rise to a weight-graded spectral sequence $E_*^{*,*}$, that is a spectral sequence where $E_r^{t,s}$ is weight-graded and the differential preserves the weight.
We have that
\[
E_1^{t,s}= \Hom(V^{-t},H^{s}(\Lambda V))\Rightarrow H^{t+s}(\Der(\Lambda V))
\]
(c.f. the proof of \cite[Lemma 3.5]{BM14}).

If $X$ is formal then $V^{-t}= \pi_{t+1}^\Q(\Omega X)^\vee$ is concentrated in weights $\geq -t$ by Proposition \ref{prop:weight<deg} \textit{\ref{item:formal}} (and by taking into account the dualization) and  $H^{s}(\Lambda V)$ is concentrated in weight $s$. Hence $\Hom(V^{-t},H^{s}(\Lambda V))$ is concentrated in weights $<s+t$. In particular $\Hom(V^{-t},H^{s}(\Lambda V))$ has negative weights if $s+t<0$, and consequently $H^*(\Der(\Lambda V)\langle -1\rangle)$ has negative weights.

If $X$ is coformal, then $V^{-t}$ is concentrated in weight $-t-1$ and $H^{s}(\Lambda V)$ is concentrated in weights $<s$ by Proposition \ref{prop:weight<deg} (b), and thus $\Hom(V^{-t},H^s(\Lambda V))$ is concentrated in weights $\leq s+t$. As before, we conclude  that $H^*(\Der(\Lambda V)\langle -1\rangle)$ has negative weights.

For the pointed version, we instead consider the spectral sequence 
$$
E_1^{t,s} = \Hom(W^{-t},H^{s}(\L W))\Rightarrow H^{t+s}(\Der(\L W)),
$$
and use the same type of arguments.
\end{proof}

\begin{exmp}
It is possible to deduce explicit bounds on the weights for the explicit dgla models for $B\aut_\circ(X)$ and $B\aut_{*,\circ}(X)$ when  $X$ is finite dimensional, simply connected and formal or coformal. We will write out these (not necessarily sharp) bounds for the dgla model for $B\aut_\circ(X)$ when $X$ is formal.

We start by observing that if $X$ is formal and simply connected then a minimal dgla model for $X$ is of the form $\L(s\bar H^*(X;\Q)^\vee)$.
In cohomological degree $n<0$ we have that
$$
\L(s\bar H^*(X)^\vee)^{n} = \sum_{\sum_{j=1}^k i_j = -n+k} S_{i_1,\dots,i_k}^n$$
where 
$$S_{i_1,\dots,i_k}^n=[s\bar H^{i_1}(X;\Q)^\vee,[s\bar H^{i_2}(X;\Q)^\vee,\cdots,[s\bar H^{i_{k-1}}(X;\Q)^\vee,s\bar H^{i_k}(X;\Q)^\vee]\cdots].
$$
Since $X$ is formal, $\bar H^{i}(X;\Q)$ is concentrated in weight $i$, and thus an element of $S_{i_1,\dots,i_k}$ is concentrated in weight $\sum_{j=1}^k(- i_j) = n-k$. Since $X$ is simply connected, the maximal value of $k$ for which $S_{i_1,\dots,i_k}^n$ is non-trivial is  $k=-n$ (where all $i_j=2$), which gives the minimal possible weight $w=2n$. The minimal value of $k$ is $1$, and in that case we have that 
\[S_{-n+1}^n = s \bar H^{-n+1}(X;\Q)^\vee\] which is concentrated in weight $(n-1)$. This is the maximal possible weight. From this we conclude that the weights in cohomological degree $n$ of the minimal dgla model for $X$ and its $n$'th cohomology $\pi_{-n}^\Q(\Omega X)$ are concentrated in the interval $[2n,n-1]$.

Hence, we have that 
\[E^{t,s}_1 = \Hom(V^{-t},\bar H^s(\Lambda V)) = \Hom(\pi_{-t-1}^\Q(\Omega X)^\vee, H^s(X;\Q))\] has weights concentrated in the interval $[s+2t+2,s+t]$.

If $X$ is  $d$-dimensional, then $H^*(X)$ is concentrated in cohomological degrees $\leq d$, and thus the spectral sequence \[E^{t,s}_1 = \Hom(V^{-t},H^s(\Lambda V))\]  associated to $\Der(\Lambda V)$ is concentrated in $s$-values satisfying $0\leq s\leq d$. In particular, for every $n\leq 0$, we have that  
$$
\bigoplus_{t+s=n} E^{t,s}_1 = E^{n,0}\oplus E^{n-1,1}\oplus\dots\oplus E^{n-d,d}     
$$
which  has weights concentrated in the interval $[2n-d+2,n]$. \end{exmp}

Theorem \ref{htpautpos} and the above example apply, for instance, to the theory of homotopy automorphisms of compact Kähler manifolds. 
In the case of complex algebraic varieties (possibly singular and/or non-projective), 
one may endow the dgla models for $B\aut_\circ(X)$ and $B\aut_{*,\circ}(X)$
with weight decompositions or, more generally, with mixed Hodge structures,
using the mixed Hodge cdga models for algebraic varieties of \cite{CG1}.
While the weights on the dgla models for homotopy automorphisms will not be negative in general,
the presence of a weight decomposition might still lead to homotopical restrictions in many situations.

\subsection{Free and based  loop spaces} 
In this last section we endow algebraic models for algebraic constructions related to the free and based loop spaces with weights inherited from weight decompositions on the original space.

The singular chains on the based loop space of a topological space $X$, denoted by $C_*(\Omega X)$, defines a dg Hopf algebra. We have that this algebra with rational coefficients is quasi-isomorphic to the universal enveloping algebra $UL$ of a dgla model $L$ for $X$. The universal enveloping algebra functor preserves weights, and so this gives a weight-graded model for $C_*(\Omega X;\Q)$ whenever $X$ has weights. Also, we have:

\begin{prop}
If $X$ has positive weights then $C_*(\Omega X;\mathbb{Q})$ has a dg associative algebra model with negative weights.
\end{prop}

The above applies, for instance, to any smooth algebraic variety.

The free loop space $LX=\map(S^1,X)$ on $X$ endowed with the compact open topology,  has  an obvious $S^1$-action \[g.\varphi(x) = \varphi(g.x)\text{ for every }g\in S^1 \text{ and }\varphi\in LX.\]
We denote the space of $S^1$-homotopy orbits by $LX//S^1$. Both $LX$ and $LX//S^1$ are  spaces of great importance in geometry for several different reasons. For instance, if the Betti numbers of the free loop space $LM$ on a manifold $M$ are unbounded then $M$ admits infinitely many geometrically distinct geodesics \cite{gromollmeyer}. The homology of $LX$ and $LX//S^1$ are also related to the theory of Hoschschild and cylic (co)homology since the (co)homology of $LX$ and $LX//S^1$ can be computed as a certain Hochschild and cyclic (co)homology. The theory of free loop spaces is also connected to homological conformal field theories.

\begin{prop}
Let $X$ be a simply connected space of finite $\Q$-type.
A     weight decomposition on an algebraic model for $X$ induces weight decompositions on the Sullivan minimal models for $LX$ and $LX//S^1$.
\end{prop}

\begin{proof} 
We will use the models for $LX$ and $LX//S^1$ constructed in \cite{burghelea-viguepoirrer}.
Let $(\Lambda V, d)$ be a Sullivan minimal model for $X$. In order to describe these models we need to set some notation. Let $\beta\in\Der(\Lambda(V\oplus s^{-1} V))$ be the unique derivation that satisfies $\beta(v)= s^{-1} v$ and $\beta(s^{-1} v)=0$. 
A Sullivan model for $LX$ is given by \[(\Lambda(V\oplus s^{-1} V), \delta)\] where $\delta$ is the unique derivation that satisfies $\delta(v)=dv$, $\delta(s^{-1} v) = -\beta dv$.
Likewise,
a model for $LX//S^1$ is given by
the algebra \[(\Lambda(V\oplus s^{-1} V\oplus \field\alpha), \mathscr D)\] where $|\alpha|=2$ and where $\mathscr D$ is the unique derivation satisfying 
\[\mathscr D(\alpha)=0,\, \mathscr D(v)= dv+\alpha\cdot s^{-1}v,\, \mathscr D(s^{-1}v)=-\beta dv\text{ for every } v\in V.\]
Given a weight decomposition for $X$, then it admits a weight graded minimal model, and hence, the above models for $LX$ and $LX//S^1$  inherit  weight decompositions, with $w(s^{-1}v)= w(v)$ and $w(\alpha)=0$.
\end{proof}

\subsection{Minimal Lie models for nilpotent spaces with positive weights}\label{existenceminimalmodels}

We end this section by proving that  nilpotent spaces 
with positive cdga models
admit minimal dgla models. In \cite{neisendorfer}, Neisendorfer proves that the rational homotopy type of a finite type nilpotent space $X$ can be modelled by a dgla $\mathscr L(\Lambda V)$, which extends the dgla models of Quillen \cite{quillenRHT} for simply connected spaces. The Neisendorfer model $\mathscr L(\Lambda V)$ is however not minimal, and Neisendorfer conjectured that minimal models for nilpotent spaces exist.

A proof of the completed version of the Neisendorfer conjecture can be found in \cite[Proposition 3.16]{bfmt}. Moreover, one can also deduce the following:

\begin{lemma}
    Given a dgla $(\widehat\L W,d)$ with a degree-wise nilpotent homology such that $d(W)\subset \L W\subseteq \widehat \L W$, then the completion map $(\L W,d)\to(\widehat \L W,d))$ is a quasi-isomorphism.
\end{lemma}
\begin{proof}
It is proved in \cite[Proposition 3.30]{bfmt} that if $(\L W,d)$ has a degree-wise nilpotent homology, then the completion map $(\L W,d)\to (\widehat\L W,d)$ is a quasi-isomorphism. 
If  $H^*(\widehat\L W,d)$ is degree-wide nilpotent, then $H^*(\L W,d)$ is also degree-wise nilpotent since 
\[H^*(\L W,d)\subseteq H^*(\widehat\L W,d).\] Hence the completion map is a quasi-isomorphism.\end{proof}

\begin{thm}\label{thm:neisendorfer}
    Let $X$ be a nilpotent space of finite type that has a cdga model with a positive weight decomposition. Then $X$ has a minimal dgla model. Moreover, the minimal model inherits a weight decomposition from the cdga model.
\end{thm}

\begin{proof}
Following \cite[§ 2]{bfmt} there is a functor $\mathscr E$ from the category of $C_\infty$-algebras with $\infty$-$C_\infty$-morphisms to the category of quasi-free dg Lie coalgebras and dg Lie coalgebra morphisms, that defines an equivalence of categories and that preserves quasi-isomorphisms. 
We briefly explain the functor $\mathscr E$: 
Let $\L^c(V)$ be a free graded Lie coalgebra  cogenerated by a finite type graded vector space $V$.  The dual $(\L^c(V))^\vee$ is given by the completed graded Lie algebra $\widehat \L(V^\vee)$

Given a unital $C_\infty$-algebra $(A,\{\mu_i\})$ we have that
$$
\mathscr E(A,\{\mu_i\}) = (\L^c(s^{-1}\overline A),\delta)
$$
where the differential decomposes as $\delta = \delta_0+\delta_1+\dots$ where $\delta_i$ is the part that decreases the word length by $i$ and where $\delta_i$ is completely encoded by $\mu_{i+1}$.

It is straightforward to show that if $(A,\{\mu_i\})$ is a  weight-graded $C_\infty$-algebra, then $\mathscr E(A,\{\mu_i\})$ is a weight-graded dg Lie coalgebra (there is a natural weight grading on $\L^c(s^{-1}\overline A)$, and we will have that $\delta_i$ preserves the weights whenever $\mu_{i+1}$ preserves the weights). Similarly one shows for any weight preserving $\infty$-$C_\infty$-morphism $g\colon A\rightsquigarrow B$, the coalgebra morphism $\mathscr E(g)$ is weight preserving. 

If $A$ is a finite type cdga, then  the dual $\mathscr E(A)^\vee$ has the structure of a dgla whose underlying graded Lie algebra structure $\widehat\L(s \overline A^\vee)$ is given by the completion of the free graded Lie algebra $\L(s \overline A^\vee)$  with respect to the bracket length.  If $A$ is a finite type $C_\infty$-algebra model for $X$, then it follows by the theory established in \cite{bfmt} that $\mathscr E(A)^\vee$ is a completed dgla model for $X$.

Now, let $(H,\{\mu_i\})$ be a weight graded minimal $C_\infty$-algebra model for the positively graded cdga model for $X$. We have \[\mathscr E(H,\{\mu_i\}) = (\L^c(s^{-1}\overline H),\delta_1+\delta_2+\dots).\]

Since $H$ is positively graded, it follows that elements of word-length $n$ in $\L^c(s^{-1} \overline H)$, i.e. elements of $\mathscr Lie(n)^\vee\ot(s^{-1}\overline H)^{\ot n}$, are of  weight at least $n$. Since $\delta_n$ vanishes on elements of word-length $<n+1$, we conclude that every non-trivial element in the image of $\delta_n$ is of at least weight $n+1$. Hence if $y\in \L^c(s^{-1} \overline H)$ and $w(y)=m$, then $y\not\in \im(\delta_i)$ for $i\geq m$

Dualizing the Lie coalgebra  $\mathcal E(H,\{\mu_i\})$  gives a complete dgla 
$$(\widehat\L(s \overline H^\vee),d_1+d_2+\dots)$$ where $d_i$ is the part of the differential that increases the bracket length by $i$.

For an element $y\in s \overline H^\vee$ of weight $w(y)=-m$, we have that $d_i(y)=0$ for every $i\geq m$.

In particular, we have that \[d(y) = d_1(y)+\dots+d_{m-1}(y)\in \L(sH^\vee)\subset \widehat\L(sH^\vee),\] and thus $\delta(s\overline H^\vee)\subset\L(s\overline H^\vee)\subset \widehat\L(s\overline H^\vee)$. Thus $\L(s\overline H^\vee,d)$ defines a minimal dgla model for $X$. 
\end{proof}

\bibliographystyle{amsalpha}
\bibliography{references}
\end{document}